\documentclass{amsart}

\usepackage{amsmath,amssymb,amsthm}
\usepackage{graphicx,tikz}

\newtheorem*{thm}{Theorem}
\newtheorem*{proposition}{Proposition}

\newtheorem{lemma}{Lemma}

\theoremstyle{definition}

\theoremstyle{remark}

\begin{document}

\title[]{On a Kantorovich-Rubinstein Inequality}
\subjclass[2010]{Kantorovich-Rubinstein, Optimal Transport.} 
\thanks{S.S. is supported by the NSF (DMS-1763179) and the Alfred P. Sloan Foundation.}

\author[]{Stefan Steinerberger}
\address{Department of Mathematics, University of Washington, Seattle, WA 98195, USA}
\email{steinerb@uw.edu}

\begin{abstract} An easy consequence of Kantorovich-Rubinstein duality is the following: if $f:[0,1]^d \rightarrow \infty$ is Lipschitz and $\left\{x_1, \dots, x_N \right\} \subset [0,1]^d$, then
$$ \left| \int_{[0,1]^d} f(x) dx - \frac{1}{N} \sum_{k=1}^{N}{f(x_k)} \right| \leq \left\| \nabla f \right\|_{L^{\infty}} \cdot W_1\left( \frac{1}{N} \sum_{k=1}^{N}{\delta_{x_k}} , dx\right),$$
where $W_1$ denotes the $1-$Wasserstein (or Earth Mover's) Distance. We prove another such inequality with a smaller norm on $\nabla f$ and a larger Wasserstein distance.  Our inequality is sharp when the points are very regular, i.e. $W_{\infty} \sim N^{-1/d}$. This prompts the question whether these two inequalities are specific instances of an entire
underlying family of estimates capturing a duality between transport distance and function space. 
\end{abstract}

\maketitle

\section{Introduction}
\subsection{Introduction} One of the most important results in Optimal Transport is Kantorovich-Rubinstein duality 
\cite{edwards, ed2, fern, hanin, kant, kell, vill1, vill2}. It states that the $1-$Wasserstein (or Earth Mover Distance) can also be defined via duality
$$ W_1(\mu, \nu) = \sup_{f~\mbox{\tiny is}~1-\mbox{\tiny Lipschitz}} ~\int_{X} f d\mu - \int_{X} f d\nu.$$
We consider this inequality in what is perhaps the simplest special case: we fix $X = [0,1]^d$, we fix $\mu = dx$ as the Lebesgue measure and define $\nu$ as the empirical distribution of $N$ given points $\left\{x_1, \dots, x_N \right\} \subset [0,1]^d$ 
$$ \nu = \frac{1}{N} \sum_{k=1}^{N} \delta_{x_k}.$$
Then Kantorovich-Rubinstein duality immediately implies a very nice interpretation of the problem of numerically integrating a function $f:[0,1]^d \rightarrow \mathbb{R}$ since
$$ \left| \int_{[0,1]^d} f(x) dx - \frac{1}{N} \sum_{k=1}^{N}{ f(x_k)} \right| \leq \| \nabla f\|_{L^{\infty}} \cdot W_1\left(dx, \frac{1}{N} \sum_{k=1}^{N}{ \delta_{x_k}} \right).$$
In particular, if we do not know anything about the function $f$ except it being Lipschitz and if we are supposed to pick the points  $\left\{x_1, \dots, x_n\right\}$, it becomes reasonable to distribute the points in such a way that the Wasserstein cost is minimized. However, we also observe that this inequality seems somewhat extremal insofar as very strict conditions are imposed on the function while the condition on the set of points is quite weak: it is measured in the $W_1$ distance which is the smallest among all Wasserstein distances. It is a natural question whether one can balance these things against each other.

\begin{quote}
\textbf{Problem.} Let $1 \leq p \leq \infty$. For which Banach spaces $X_p$ 
$$ \left| \int_{[0,1]^d} f(x) dx - \frac{1}{N} \sum_{k=1}^{N}{ f(x_k)} \right| \leq c_d \cdot \| \nabla f\|_{X_p} \cdot W_p\left(dx, \frac{1}{N} \sum_{k=1}^{N}{ \delta_{x_k}} \right)?$$
\end{quote}

Certainly, there is such a canonical Banach space for $p=1$ and, by Kantorovich-Rubinstein duality, we have $X_1 = L^{\infty}$. Moreover, since the Wasserstein distance is an increasing quantity in $p$, we see that $X_p = L^{\infty}$ is always an admissible choice. The question is: is it possible to replace it by another Banach space when $p > 1$? In trying to investigate this question, the case $p= \infty$ is a natural end-point. 
\begin{quote}
\textbf{Question.} Is
$ X_{\infty} = L^{d,1}$ an admissible choice for $p=\infty$?
\end{quote}
$L^{d,1}$ is the Lorentz space refinement of classical Lebesgue spaces, i.e.
$$ \|f\|_{L^{d,1}} = d \cdot \int_{0}^{\infty}  \left| \left\{ x: |f(x)| \geq t \right\}\right|^{\frac{1}{d}} dt.$$
If $X_{\infty} = L^{d,1}$ was an admissible choice, then we can approximate any given measure $\mu$
arbitrarily well by discrete measures and would arrive at an inequality of the type
$$ \left| \int_{[0,1]^d} f(x) d\mu  -  \int_{[0,1]^d} f(x) dx  \right| \leq c_d \cdot \| \nabla f\|_{L^{d, 1}} \cdot W_{\infty}\left(\mu, dx \right).$$
Kantorovich-Rubinstein duality is considerably more general since it deals with two arbitrary measures while we require one of the measures to be the Lebesgue measure $\nu = dx$. However, it is relatively easy to see that if both measures are allowed to be singular, one cannot get a better bound than $\| \nabla f\|_{L^{\infty}}$: pick $\mu$ and $\nu$ to be two Dirac measures, the transport cost is determined by the behavior of the function on the line segment connecting the two points.

\subsection{The Result.} We now present our result which indicates that $X_{\infty} = L^{d, 1}$ could be a reasonable guess. Moreover, our result is actually optimal in the endpoint where points are as regularly distributed as a grid, i.e. $W_{\infty} \sim N^{-1/d}$.

\begin{thm} For any $f:[0,1]^d \rightarrow \mathbb{R}$ and any $\left\{x_1, \dots, x_N\right\} \subset [0,1]^d$,
$$ E = \left| \int_{[0,1]^d} f(x) dx - \frac{1}{N} \sum_{k=1}^{N}{ f(x_k)} \right| $$
is bounded from above by
$$E \leq c_d    \cdot  \|\nabla f\|_{L^{\infty}([0,1]^d)}^{\frac{d-1}{d}} \cdot \|\nabla f\|_{L^{1}([0,1]^d)}^{\frac{1}{d}}  \cdot N^{1/d} \cdot W_{\infty}\left(dx, \frac{1}{N} \sum_{k=1}^{N}{\delta_{x_k}} \right)^2.$$
\end{thm}
\textbf{Remarks.} Several remarks (and an explanation as to in what sense we consider the result `close' to the conjectured result) are in order. 
\begin{enumerate}
\item The inequality is sharp for any set of points satisfying $W_{\infty} \sim N^{-1/d}$. For any such set, an extremal example can be taken, for $\varepsilon$ sufficiently small, as 
$$ f_{\varepsilon}(x) = \min\left\{ \varepsilon, \min_{1 \leq k \leq N} \| x - x_k\| \right\}.$$
As soon as $0 < \varepsilon \ll N^{-1/d}$, we have $E \sim \varepsilon$ as well as 
$$\|\nabla f_{\varepsilon}\|_{L^{\infty}([0,1]^d)} \sim 1 \quad \mbox{and} \quad
 \|\nabla f_{\varepsilon}\|_{L^{1}([0,1]^d)} \sim \varepsilon^d \cdot N.$$ 
\item We note the interpolation inequality (see Lemma 2)
$$ \left\| g \right\|_{L^{d, 1}([0,1]^d)} \lesssim_d \|g\|_{L^{\infty}([0,1]^d)}^{\frac{d-1}{d}} \cdot \| g\|_{L^{1}([0,1]^d)}^{\frac{1}{d}}.$$
For the extremal function $f_{\varepsilon}$, both sides are comparable.
\item As for the remaining term, we note that 
$$    W_{\infty}\left(dx, \frac{1}{N} \sum_{k=1}^{N}{\delta_{x_k}} \right) \lesssim N^{1/d} \cdot W_{\infty}\left(dx, \frac{1}{N} \sum_{k=1}^{N}{\delta_{x_k}} \right)^2$$
again with equality if $W_{\infty}$ is as small as possible (i.e. $W_{\infty} \sim N^{-1/d}$). 
\item The restriction to $[0,1]^d$ is to simplify comparison with existing results, one would naturally
expect the result to hold for fairly general domains.
\end{enumerate}

There are various intermediate results that lie between our Theorem and the conjectured result $X_{\infty} = L^{d,1}$. For example, one could ask whether there exists $0 < \delta < 1$ such that, with $E$ playing the same role as in the Theorem,
$$E \leq c_d    \cdot  \|\nabla f\|_{L^{\infty}([0,1]^d)}^{\frac{d-1}{d}} \cdot \|\nabla f\|_{L^{1}([0,1]^d)}^{\frac{1}{d}}  \cdot N^{\delta/d} \cdot W_{\infty}\left(dx, \frac{1}{N} \sum_{k=1}^{N}{\delta_{x_k}} \right)^{1+\delta}.$$
We are proving the estimate for $\delta = 1$. The smaller $\delta$, the harder the statement. In \S 4.2, we sketch a relatively simple proof for $\delta = d$ which avoids Lemma 1.

\subsection{A Lemma.} Most of the actual argument goes towards establishing an isoperimetric Lemma which seems like it might be of interest in its own right.

\begin{lemma} Let $\mu$ be a measure on $\mathbb{R}^d$ such that
\begin{enumerate}
\item $\mu$ is compactly supported in a ball of radius $R$ around the origin
\item $\mu$ is absolutely continuous and $\mu \leq dx$.
\end{enumerate}
Then, for all $f:\mathbb{R}^d \rightarrow \mathbb{R}$ such that $f(0)=0$, we have
$$ \left| \int_{\mathbb{R}^d} f(x) d\mu \right| \leq c_d\cdot  R \cdot \mu(\mathbb{R}^d)^{\frac{d-1}{d}} \cdot \left\| \nabla f\right\|_{L^{d, 1}(\|x\| \leq R)}$$
\end{lemma}

 To illustrate the Lemma, we consider two explicit examples which show that the Lemma is optimal in different regimes. For simplicity of exposition, we use $\sim$ to denote equivalence up to constants depending only on the dimension. The first example is as follows:
let $f:\mathbb{R}^d \rightarrow \mathbb{R}$ be given by
$$ f(x) = \begin{cases} \|x\|  \qquad &\mbox{if}~\|x\| \leq \delta \\ \delta \qquad &\mbox{otherwise,}\end{cases}$$
we let $\mu$ be the Lebesgue measure on the ball centered at 0 having total volume $\mu(\mathbb{R}^d)$ and let us assume that $0 < \delta \ll  \mu(\mathbb{R}^d)^{1/d}$. Then $R \sim \mu(\mathbb{R}^d)^{1/d}$ and
$$ \left| \int_{\mathbb{R}^d} f(x) d\mu \right| \sim \delta \cdot \mu(\mathbb{R}^d).$$
We also have
\begin{align*}
 \| \nabla f\|_{L^{d,1}} = d \int_0^{\infty} \left| \left\{x: \| \nabla f(x) \| \geq t \right\} \right|^{1/d} dt \sim \int_0^{1} \delta~ dt =  \delta.
 \end{align*}
 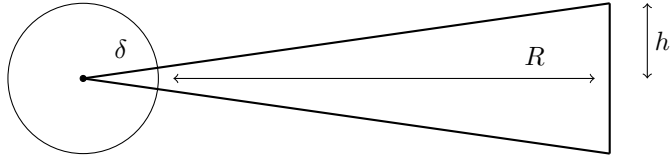
\begin{figure}[h!]
 \begin{center}
 \begin{tikzpicture}
 \filldraw (0,0) circle (0.04cm);
 \draw [thick] (0,0) -- (7,1);
 \draw [thick] (0,0) -- (7,-1);
 \draw [<->] (1.2, 0) -- (6.8,0);
 \draw [thick] (7,-1)-- (7,1);
 \node at (6, 0.3) {$R$};
 \draw [<->] (7.5, 0) -- (7.5, 1);
 \node at (7.7, 0.5) {$h$};
 \draw (0,0) circle (1cm);
 \node at (0.5, 0.4) {$\delta$};
 \end{tikzpicture}
 \end{center}
 \caption{Geometric setup for the second example}
 \label{fig:cone}
 \end{figure}
 
  The same function can be used to give a geometrically more interesting example (see Fig. \ref{fig:cone}): fix $0 < \delta \ll 1$, $R \gg 1, h \ll R$ and let $\mu$ denote the Lebesgue measure in a cone centered
 in 0 having diameter $R$ and let $f$ be the same function as above.
  Then, since $\delta \ll R$,
 $$ \left| \int_{\mathbb{R}^d} f(x) d\mu \right| \sim \delta \cdot \mu(\mathbb{R}^d).$$
 We also have
$$  \| \nabla f\|_{L^{d,1}} = d \int_0^{\infty} \left| \left\{x: \| \nabla f(x) \| \geq t \right\} \right|^{1/d} dt \sim\left( \delta^d \cdot \frac{h^{d-1}}{R^{d-1}}\right)^{\frac{1}{d}}.$$
 Using the relationship $R \cdot h^{d-1} \sim \mu(\mathbb{R}^d)$, we have
 \begin{align*}
R \cdot \mu(\mathbb{R}^d)^{\frac{d-1}{d}} \cdot  \left\| \nabla f\right\|_{L^{d, 1}} &\sim R \cdot \mu(\mathbb{R}^d)^{\frac{d-1}{d}} \cdot    \left( \delta^d \cdot \frac{h^{d-1}}{R^{d-1}}\right)^{\frac{1}{d}} \\
 &= R \cdot \mu(\mathbb{R}^d)^{\frac{d-1}{d}} \cdot  \left( \delta^d \cdot \frac{R h^{d-1}}{R^{d}}\right)^{\frac{1}{d}}\\
 &=  R \cdot \mu(\mathbb{R}^d)^{} \cdot  \frac{\delta}{ R} \sim  \left| \int_{\mathbb{R}^d} f(x) d\mu \right|.
 \end{align*}

\subsection{Related Results.}
It is classical that there exist $\left\{x_1, \dots, x_N\right\} \subset [0,1]^d$ such that for all Lipschitz $f:[0,1]^d \rightarrow \mathbb{R}$
$$ \left| \int_{[0,1]^d} f(x) dx - \frac{1}{N} \sum_{k=1}^{N}{ f(x_k)} \right| \leq c_d \frac{\| \nabla f\|_{L^{\infty}}}{N^{1/d}}$$
and that this is best possible. The result is often ascribed to Bakhalov \cite{bak}. There has been a lot of work on this problem, especially with regards
to how the implicit constant depends on the dimension $d$ and the smoothness (which we here fix to be $r=1$, we only consider one derivative). We
refer to \cite{hin, hin2, hinrichs, leo, novak, weed} for some recent results.                                                                                                                                                                                                                                                                                                                                                                                                                                                                                                                                                                                                                                                                                                                                                                                                                                                                                                                                                                                                                                                                                                                                                                                                                                                                                                                                                                                                                                                                                                                                                                                                                                                                                                                                                                                                                                                                                                                                                                                                                                                                                                                                                                                                                                                                                                                 
It seems to have been pointed out only rather recently \cite{brown} that, in fact, there exist $\left\{x_1, \dots, x_N\right\} \subset [0,1]^d$ such that for all Lipschitz $f:[0,1]^d \rightarrow \mathbb{R}$
$$ \left| \int_{[0,1]^d} f(x) dx - \frac{1}{N} \sum_{k=1}^{N}{ f(x_k)} \right| \leq c_d \frac{ \|\nabla f\|_{L^{\infty}([0,1]^d)}^{\frac{d-1}{d}} \cdot \|\nabla f\|_{L^{1}([0,1]^d)}^{\frac{1}{d}}}{N^{1/d}}.$$
The example in \cite{brown} is a regular grid -- as a consequence of our main result in this paper, we have the same estimate for any set of points that satisfy
$$ W_{\infty}\left(dx, \frac{1}{N} \sum_{k=1}^{N}{\delta_{x_k}} \right) \sim N^{-1/d}.$$
We also observe some vague similarity to recent results on Sobolev-Kantorovich inequalities. Cinti and Otto \cite{cinti} showed that for some $c_d > 0$ depending only on the dimension and $f:[0,1]^d \rightarrow \mathbb{R}$ normalized to $\int_{[0,1]^d} f(x) dx = 1$,
$$ \left\| \max \left\{ f - c_d, 0 \right\} \right\|_{L^{1+\frac{2}{3d}}}^{1 + \frac{2}{3d}} \lesssim \| \nabla f\|_{L^1} \cdot W_2(f, dx).$$
This has then been generalized by Ledoux \cite{ledoux} who showed that for any $p,q$, there exist $c_{p,q,d} > 0$ as well as $r, \theta$ such that
$$ \left\| \max \left\{ f - c_{p,q,d}, 0 \right\} \right\|_{L_{r}}^{\theta} \lesssim \| \nabla f\|_{L^q} \cdot W_p(f, dx).$$

\section{Three Lemmata}

\subsection{An Isoperimetric Lemma} In this section, we prove Lemma 1: what is interesting is that both the function $f$ and the measure $\mu$ may vary.\\

\textbf{Lemma 1.} \textit{ Let $\mu$ be a measure on $\mathbb{R}^d$ such that
\begin{enumerate}
\item $\mu$ is compactly supported in a ball of radius $R$ around the origin
\item $\mu$ is absolutely continuous and $\mu \leq dx$.
\end{enumerate}
Then, for all $f:\mathbb{R}^d \rightarrow \mathbb{R}$ such that $f(0)=0$, we have
$$ \left| \int_{\mathbb{R}^d} f(x) d\mu \right| \leq c_d\cdot  R \cdot \mu(\mathbb{R}^d)^{\frac{d-1}{d}} \cdot \left\| \nabla f\right\|_{L^{d, 1}(\|x\| \leq R)}.$$
}
\vspace{-10pt}
\begin{proof} 
Since $f(0) = 0$, we can use the fundamental theorem of calculus, the triangle inequality and the Cauchy-Schwarz inequality to argue that
\begin{align*}
 |f(x)| &= \left|    \int_{0}^{\|x \|}   \left\langle \nabla f \left( t \frac{x}{\|x\|}\right),  \frac{x}{\|x\|} \right\rangle dt  \right| \\
 &\leq\int_{0}^{\|x \|}     \left|     \left\langle \nabla f \left( t \frac{x}{\|x\|}\right),  \frac{x}{\|x\|} \right\rangle \right| dt   \leq \int_{0}^{\|x \|} \left|  \nabla f \left( t \frac{x}{\|x\|}\right) \right| dt.
\end{align*}
We will only work with this upper bound and will show that
$$ \int_{\mathbb{R}^d}  \int_{0}^{\|x \|} \left|  \nabla f \left( t \frac{x}{\|x\|}\right) \right| dt d\mu \leq  c_d\cdot  R \cdot \mu(\mathbb{R}^d)^{\frac{d-1}{d}} \cdot \left\| \nabla f\right\|_{L^{d, 1}(\|x\| \leq R)}.$$

 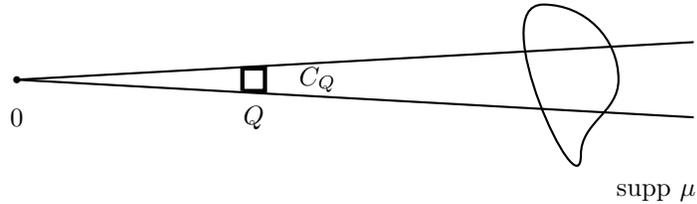
\begin{figure}[h!]
 \begin{center}
 \begin{tikzpicture}
 \filldraw (0,0) circle (0.04cm);
 \node at (0,-0.5) {0};
 \draw [thick] (0,0) -- (9,0.5);
 \draw [thick] (0,0) -- (9,-0.5);
 \draw [thick] (8,0) to[out=90, in=0] (7, 1) to[out=180, in =270] (7.5, -1) to[out=90, in =270] (8,0);
 \node at (8.5, -1.5) {$\mbox{supp} ~\mu$};
 \draw [ultra thick] (3,-0.15) -- (3,0.15) -- (3.3, 0.15) -- (3.3, -0.15) -- (3,-0.15);
 \node at (3.15, -0.5) {$Q$};
 \node at (4, 0) {$C_Q$};
 \end{tikzpicture}
 \end{center}
 \caption{Geometric Interpretation of the measure $\nu$.}
\label{fig:nu}
 \end{figure}
 We start by exchanging the order of integration and end up with 
$$ \int_{\mathbb{R}^d}  \int_{0}^{\|x \|} \left|  \nabla f \left( t \frac{x}{\|x\|}\right) \right| dt d\mu = \int_{\mathbb{R}^d} | \nabla f(x)| d\nu(x),$$
where $\nu$ is the measure that arises from integrating over all the line integrals weighted by $\mu$. $\nu$ has a simple geometric interpretation: for any point $x$, the measure $\nu$ simply counts the total amount of $\mu$ in the induced cone (see Fig. \ref{fig:nu}).
 More precisely, for the type of functions 
 $$g(x) = \int_{0}^{\|x \|} \left|  \nabla f \left( t \frac{x}{\|x\|}\right) \right| dt$$
  under consideration, we have
 $$  \int_{\mathbb{R}^d} g(x) d\mu(x) = \int_{\mathbb{R}^d} | \nabla f(x)| d\nu(x).$$
Thus, if $Q$ is an infinitesimal square (as in Fig. 2), the density of $\nu$ at the point $Q$ can be evaluated by a simple trick. Suppose we increase the size of the gradient $\| \nabla f\|$ by $\delta$ within $Q$. This increases the integral on the right by $\delta \cdot \nu(Q)$ while it increases the function $f$ by $\delta \cdot |Q|^{1/d}$ on the entire cone $C_Q$ after $Q$. Thus, for infinitesimal cubes $Q$
$$ \nu(Q) = |Q|^{1/d} \mu(C_Q).$$
For any given $\mu$ and all such functions $g$, we will now construct a measure $\mu_2$ such that $\mu_2(\mathbb{R}^d) = \mu(\mathbb{R}^d)$ and
 $$ \int_{\mathbb{R}^d} g(x) d\mu \leq  \int_{\mathbb{R}^d} g(x) d\mu_2.$$ 
$\mu_2$ will be a measure supported on $\left\{x \in \mathbb{R}^d: \|x\| = R \right\}$. For $A \subset \left\{x \in \mathbb{R}^d: \|x\| = R \right\}$, we define the cone 
$$ C_A = \left\{ \lambda x: 0 \leq \lambda \leq 1 \wedge x \in A \right\}$$
and set
$$ \mu_2(A) = \mu(C_A).$$
Since $\mu \leq dx$, we get that $\mu_2$ is absolutely continuous with respect to the surface measure $\sigma$ on $\left\{x \in \mathbb{R}^d: \|x\| = R\right\}$ and denote the Radon-Nikodym derivative by
$$ \phi = \frac{d \mu_2}{d \sigma}.$$
Note that we have $\| \phi\|_{L^{\infty}} \lesssim R$ and
$$ \int_{\|x\| = R}{ \phi ~ d\sigma } = \mu_2(\mathbb{R}^d) = \mu(\mathbb{R}^d).$$
By the same construction as above, we have
$$ \int_{\mathbb{R}^d} g d\mu_2 = \int_{\mathbb{R}^d} | \nabla f| d\nu_2$$
and $\nu_2$ is determined exactly as $\nu$ is.  However, since $\mu_2$ is much simpler than $\mu$, we can give an explicit expression for $\nu_2$: we claim that $\nu_2$ is absolutely
continuous with density
$$ \nu_2 = c_d R^{d-1} \frac{\phi \left(\frac{x}{\|x\|} R\right)}{ \|x\|^{d-1}} dx.$$
This can be seen as follows: let $Q$ be an infinitesimal cube centered at $x$ (oriented as in Fig. 2). We have, as $\mbox{diam}(Q) \rightarrow 0$ that
$$ \nu_2(Q) = |Q|^{1/d} \mu_2(C_Q).$$
It thus remains to determine $\mu_2(C_Q)$: clearly, the only relevant quantity is the intersection of $C_Q$ with $\left\{x \in \mathbb{R}^d: \|x\| = R\right\}$. The density
is given by the Radon-Nikodym derivative at $R x/\|x\|$. As for the surface area, we have that, as $|Q| \rightarrow 0$,
$$ \sigma(C_Q) = c_d R^{d-1} \frac{|Q|^{\frac{d-1}{d}}}{\|x\|^{d-1}}$$
and thus
$$ \mu_2(C_Q) = \phi\left(R \frac{x}{\|x\|}\right) \sigma(C_Q) = c_d R^{d-1}  \phi\left(R \frac{x}{\|x\|}\right) \frac{|Q|^{\frac{d-1}{d}}}{\|x\|^{d-1}}$$
from which we deduce
$$ \lim_{|Q| \rightarrow 0}{ \frac{\nu_2(Q)}{|Q|}} = c_d R^{d-1}  \frac{ \phi\left(R \frac{x}{\|x\|}\right)}{\|x\|^{d-1}}$$
Altogether
\begin{align*}
 \int_{\mathbb{R}^d} g d\mu \leq  \int_{\mathbb{R}^d} g d\mu_2 &= \int_{\mathbb{R}^d} | \nabla f| d\nu_2 \\
 &= c_d \int_{ \left\{ x \in \mathbb{R}^d: \|x\| \leq R \right\}} | \nabla f(x) | R^{d-1} \frac{\phi \left(\frac{x}{\|x\|} R\right)}{ \|x\|^{d-1}} dx.
 \end{align*}
Using the duality in Lorentz spaces  (an inequality of O'Neil \cite{oneil}), we have
$$ \int_{ \left\{ x \in \mathbb{R}^d: \|x\| \leq R \right\}} | \nabla f (x) |  R^{d-1} \frac{\phi \left(\frac{x}{\|x\|} R\right)}{ \|x\|^{d-1}} dx \leq \| \nabla f\|_{L^{d,1}} \cdot 
\left\|R^{d-1}  \frac{\phi \left(\frac{x}{\|x\|} R\right)}{ \|x\|^{d-1}} \right\|_{L^{\frac{d}{d-1}, \infty}}.$$
It remains to bound the second norm from above. We recall the definition of the Lorentz space, 
$$ \|h\|_{L^{\frac{d}{d-1}, \infty}}^{\frac{d}{d-1}} = \sup_{t>0} ~   t^{\frac{d}{d-1}} \cdot \left| \left\{x: |h(x)| \geq t \right\} \right|^{}.$$
Let us now fix any value $t>0$. We will compute the volume of the super-level set by switching to spherical coordinates. In direction $x$, we have
$$ R^{d-1}  \frac{\phi \left(\frac{x}{\|x\|} R\right)}{ \|x\|^{d-1}} \geq t \qquad \mbox{iff} \qquad \|x\| \leq   R\frac{\phi \left(\frac{x}{\|x\|} R\right)^{\frac{1}{d-1}}}{t^{\frac{1}{d-1}}}.$$
Therefore, changing to spherical coordinates and recalling $ \|\phi\|_{L^{\infty}} \lesssim_d R$,
\begin{align*}  \left| \left\{x: \frac{\phi \left(\frac{x}{\|x\|} R\right)}{ \|x\|^{d-1}} \geq t \right\} \right| &\lesssim_d \int_{ \mathbb{S}^{d-1}} \left(  R\frac{\phi \left(\frac{x}{\|x\|} R\right)^{\frac{1}{d-1}}}{t^{\frac{1}{d-1}}} \right)^d d\sigma(x) \\
&\lesssim_d \frac{R}{t^{\frac{d}{d-1}}}\int_{\|x\|=R}  \phi \left(x\right)^{\frac{d}{d-1}} d\sigma(x) \\
&\leq  \frac{R \cdot \| \phi\|_{L^{\infty}}^{\frac{1}{d-1}} }{t^{\frac{d}{d-1}}}  \int_{\|x\|=R}  \phi \left(x\right)^{} d\sigma(x)\\
&= \frac{R \cdot \| \phi\|_{L^{\infty}}^{\frac{1}{d-1}} }{t^{\frac{d}{d-1}}}  \mu(\mathbb{R}^d) \leq \frac{R^{\frac{d}{d-1}}}{t^{\frac{d}{d-1}}} \mu(\mathbb{R}^d).
\end{align*}
Thus
$$ \left\| R^{d-1} \frac{\phi \left(\frac{x}{\|x\|} R\right)}{ \|x\|^{d-1}} \right\|_{L^{\frac{d}{d-1}, \infty}} \lesssim_d  R \cdot \mu(\mathbb{R})^{\frac{d-1}{d}}$$
which is the desired statement.
\end{proof}

\subsection{Comparing spaces}
The purpose of this short section is to establish a simple interpolation Lemma. It is not new, very simple and has been stated many times in the literature, we include it for the convenience of the reader.
\begin{lemma} For any subset $X \subset [0,1]^d$, we have
$$ \|  h\|_{L^{d,1}(X)} \leq  d\cdot \| h\|_{L^{\infty}(X)}^{\frac{d-1}{d}} \| h \|_{L^1(X)}^{\frac{1}{d}}.$$
\end{lemma}
\begin{proof}
\begin{align*}
\|  h\|_{L^{d,1}(X)} &= d\cdot \int_{0}^{\infty} \left|\left\{ x: \left| h(x) \right| > \lambda \right\}\right|^{1/d} d \lambda \\
&=  d\cdot \int_{0}^{\| h \|_{L^{\infty}}} \left|\left\{x: \left| h(x)\right| > \lambda \right\}\right|^{1/d} d \lambda \\
&\leq d\cdot \| h \|_{L^{\infty}}^{\frac{d-1}{d}} \cdot \left( \int_{0}^{\infty} \left|\left\{x: |h(x)| > \lambda \right\}\right| d\lambda \right)^{\frac{1}{d}} = d\cdot \| h\|_{L^{\infty}}^{\frac{d-1}{d}} \| h\|_{L^1}^{\frac{1}{d}}.
\end{align*}
\end{proof}

\subsection{Controlling the density}
Finally, we argue that if 
$$ W_{\infty}\left( \frac{1}{N} \sum_{k=1}^{N}{\delta_{x_k}} , dx\right) \qquad \mbox{is small,}$$
then this means that the points cannot be arbitrarily distributed: in particular, no ball of radius $W_{\infty}$ can contain a disproportionate number of points.
\begin{lemma} Let $\left\{x_1, \dots, x_N\right\} \subset [0,1]^d$ and let us abbreviate
$$ W_{\infty} = W_{\infty}\left( \frac{1}{N} \sum_{k=1}^{N}{\delta_{x_k}} , dx\right).$$
 Then, for any $x \in [0,1]^d$ and some universal constant $c_d$ depending only on the dimension,
$$ \# \left\{1 \leq i \leq N: \|x_i - x\| \leq  W_{\infty} \right\} \leq c_d \cdot W_{\infty}^{d} \cdot N.$$
\end{lemma}
\begin{proof} Let $x \in [0,1]^d$ and suppose there are $X-1$ other points at distance at most $W_{\infty}$. There is an optimal transport plan such that neither of these points has
to transport their mass further than $W_{\infty}$, their total mass is thus contained in a $2\cdot W_{\infty}-$ball around $x$. However, the total amount of mass ending up in this ball is controlled and thus
$$ X \cdot \frac{1}{N} \leq \left| B(x,2W_{\infty}) \cap [0,1]^d \right| \leq  \left| B(x,2W_{\infty})  \right| \leq c_d\cdot W_{\infty}^d$$
from which the desired bound follows.
\end{proof}

\section{Proof of the Theorem}
\begin{proof} We can now combine the various Lemmata to obtain a proof. We again abbreviate, for simplicity of exposition,
$$ W_{\infty} = W_{\infty}\left( \frac{1}{N} \sum_{k=1}^{N}{\delta_{x_k}} , dx\right).$$
Denoting the region where the mass in $x_k$ is being transported to by $X_k$, we have
$$X_k \subset B(x_k, W_{\infty}) \cap [0,1]^d.$$ The triangle inequality combined with Lemma 1 yields
\begin{align*}
  \left| \int_{[0,1]^d} f(x) dx - \frac{1}{N} \sum_{k=1}^{N}{f(x_k)} \right| &\leq \sum_{k=1}^{N} \left| \int_{X_k} f(x) d\mu_k - \frac{f(x_k)}{N} \right| \\
  &\leq \sum_{k=1}^{N} \left| \int_{X_k} f(x) - f(x_k) d\mu_k \right| \\
  &\leq \sum_{k=1}^{N} \frac{W_{\infty}}{N^{\frac{d-1}{d}}} \cdot \| \nabla f\|_{L^{d,1}(X_k)}.
\end{align*}
Lemma 2 leads to the upper bound
 $$ \sum_{k=1}^{N} \frac{W_{\infty}}{N^{\frac{d-1}{d}}} \cdot \| \nabla f\|_{L^{d,1}(X_k)} \leq d  \frac{W_{\infty}}{N^{\frac{d-1}{d}}} \sum_{k=1}^{N} \|\nabla f\|_{L^{\infty}(X_k)}^{\frac{d-1}{d}} \|\nabla f\|_{L^1(X_k)}^{\frac{1}{d}}.$$
We obtain a further bound from above by setting
$$  \frac{W_{\infty}}{N^{\frac{d-1}{d}}} \sum_{k=1}^{N} \|\nabla f\|_{L^{\infty}(X_k)}^{\frac{d-1}{d}} \|\nabla f\|_{L^1(X_k)}^{\frac{1}{d}} \leq
 \frac{W_{\infty}}{N^{\frac{d-1}{d}}}  \|\nabla f\|_{L^{\infty}([0,1]^d)}^{\frac{d-1}{d}} \sum_{k=1}^{N} \|\nabla f\|_{L^1(X_k)}^{\frac{1}{d}}.$$
Applying H\"older's inequality results in 
$$  \frac{W_{\infty}}{N^{\frac{d-1}{d}}}  \|\nabla f\|_{L^{\infty}([0,1]^d)}^{\frac{d-1}{d}} \sum_{k=1}^{N} \|\nabla f\|_{L^1(X_k)}^{\frac{1}{d}} 
\leq  W_{\infty}  \|\nabla f\|_{L^{\infty}([0,1]^d)}^{\frac{d-1}{d}} \left( \sum_{k=1}^{N} \|\nabla f\|_{L^1(X_k)} \right)^{\frac{1}{d}}.$$
Finally, Lemma 3 guarantees that the different regions cannot over-count too much and
$$  \sum_{k=1}^{N} \|\nabla f\|_{L^1(X_k)} \leq   W_{\infty}^d \cdot N \cdot \| \nabla f\|_{L^1([0,1]^d)}.$$
From this we obtain the desired result.
\end{proof}

\section{Concluding Remarks}
\subsection{The missing step.} It is presumably the case that the argument is lossy. We believe that the crucial part is the following: if $W_{\infty} \sim N^{-1/d}$, then each point is transported to a nearby area. In particular, there is relatively little overlap between the transport: any given tiny area will not be traversed by many transport plans. However, if $W_{\infty}$ becomes bigger, this is harder to guarantee.
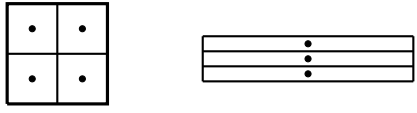
\begin{figure}[h!]
\begin{center}
\begin{tikzpicture}[scale=2]
\draw [very thick] (0,0) -- (2/3,0) -- (2/3,2/3) -- (0,2/3) -- (0,0);
\filldraw (1/6, 1/6) circle (0.02cm);
\filldraw (3/6, 1/6) circle (0.02cm);
\filldraw (1/6, 3/6) circle (0.02cm);
\filldraw (3/6, 3/6) circle (0.02cm);
\draw [thick] (1/3,0) -- (1/3,2/3);
\draw [thick] (2/3,0) -- (2/3,2/3);
\draw [thick] (0,1/3) -- (2/3,1/3);
\draw [thick] (0,2/3) -- (2/3,2/3);
\filldraw (2, 0.2) circle (0.02cm);
\filldraw (2, 0.3) circle (0.02cm);
\filldraw (2, 0.4) circle (0.02cm);
\draw [thick] (1.3, 0.15) -- (2.7, 0.15);
\draw [thick] (1.3, 0.25) -- (2.7, 0.25);
\draw [thick] (1.3, 0.35) -- (2.7, 0.35);
\draw [thick] (1.3, 0.45) -- (2.7, 0.45);
\draw [thick] (1.3, 0.15) -- (1.3, 0.45);
\draw [thick] (2.7, 0.15) -- (2.7, 0.45);
\end{tikzpicture}
\end{center}
\caption{Left: well-separated points and the regions they transport to. Right: more clustered points and the same regions.}
\label{fig:arr}
\end{figure}

This leads to an interesting question, one that would be implicitly answered by an estimate of the flavor
$$ \left| \int_{X} f(x) d\mu  -  \int_{X} f(x) dx  \right| \leq c_d \cdot \| \nabla f\|_{L^{d, 1}} \cdot W_{\infty}\left(\mu, dx \right),$$
is whether this remains true if some mass is transported a great distance. Are there particularly `busy intersections', small regions in space that are traversed by a large amount of measure in roughly the same direction?

\subsection{A cheap argument.} We conclude by showing a very cheap version of the argument which avoids Lemma 1 and leads to a result that is always weaker than our Theorem but just
as strong in the endpoint $W_{\infty} \sim N^{-1/d}$. We can think of $W_{\infty} \sim N^{-1/d}$ as both the best possible case but also as the case where the actual transport behavior is relatively simple. This result also illustrates that the difficulty is in not losing too many powers in $N$ along the way.

\begin{proposition}
For any Lipschitz $f:[0,1]^d \rightarrow \mathbb{R}$ and any $\left\{x_1, \dots, x_N\right\} \subset [0,1]^d$,
$$ E = \left| \int_{[0,1]^d} f(x) dx - \frac{1}{N} \sum_{k=1}^{N}{ f(x_k)} \right| $$
is bounded from above by
$$E \leq c_d    \cdot  \|\nabla f\|_{L^{\infty}([0,1]^d)}^{\frac{d-1}{d}} \cdot \|\nabla f\|_{L^{1}([0,1]^d)}^{\frac{1}{d}}  \cdot N^{} \cdot W_{\infty}\left(dx, \frac{1}{N} \sum_{k=1}^{N}{\delta_{x_k}} \right)^{d+1}.$$
\end{proposition}

We need a different ingredient which is somewhat standard and included for the convenience of the reader. It was also used, for example, in \cite{brown}.
\begin{lemma} Let $B(0,r)$ be a ball of radius $r$ and let $f:B(0,r) \rightarrow \mathbb{R}$ be a Lipschitz function that vanishes in the origin, $f(0) = 0$. Then
$$ \left| \int_{B(0,r)} f(x) dx \right| \leq c_d \cdot r^d \cdot \| \nabla f\|_{L^{\infty}(B)}^{\frac{d-1}{d}} \| \nabla f \|_{L^1(B)}^{\frac{1}{d}}.$$
\end{lemma}
\begin{proof} By scaling, it suffices to consider the unit ball $B = B(0,1)$.  Using an argument that is often employed in the proof of Morrey's inequality (see Evans \cite[\S 5.6.2]{evans})
$$ \left| \int_{B} f(x) dx \right| \leq c_d \int_{B} \frac{|\nabla f|}{\|x\|^{d-1}} dx.$$
Using O'Neil's inequality \cite{oneil} and Lemma 2, we obtain
$$ \left| \int_{B} f(x) dx  \right| \lesssim_d \left\| \nabla f \right\|_{L^{d,1}(B)} \lesssim_d  \| \nabla f\|_{L^{\infty}(B)}^{\frac{d-1}{d}} \| \nabla f \|_{L^1(B)}^{\frac{1}{d}}$$
which is the desired result.
\end{proof}

\begin{proof}[Proof of the Proposition]  We argue as above and obtain 
\begin{align*}
  \left| \int_{[0,1]^d} f(x) dx - \frac{1}{N} \sum_{k=1}^{N}{f(x_k)} \right| &\leq \sum_{k=1}^{N} \left| \int_{X_k} f(x) d\mu_k - \frac{f(x_k)}{N} \right| \\
  &\leq \sum_{k=1}^{N} \left| \int_{X_k} f(x) - f(x_k) d\mu_k \right|.
\end{align*}
We know, from Lemma 3, that $ X_k \subset B(x_k, W_{\infty})$ and we also know that $\mu_k \leq dx$. Therefore
$$  \sum_{k=1}^{N} \left| \int_{X_k} f(x) - f(x_k) d\mu_k \right| \leq  \sum_{k=1}^{N} \left| \int_{B(x,W_{\infty})} f(x) - f(x_k) dx \right|.$$
Using Lemma 4 and H\"older's inequality results in
\begin{align*}  \sum_{k=1}^{N} \left| \int_{B(x,W_{\infty})} f(x) - f(x_k) dx \right| &\lesssim   \sum_{k=1}^{N} W_{\infty}^d \cdot \| \nabla f\|_{L^{\infty}(B(x_k, W_{\infty})}^{\frac{d-1}{d}} \| \nabla f \|_{L^1(B(x_k, W_{\infty}))}^{\frac{1}{d}} \\
&\lesssim W_{\infty}^d \cdot \| \nabla f\|_{L^{\infty}}^{\frac{d-1}{d}}  \cdot \sum_{k=1}^{N}   \| \nabla f \|_{L^1(B(x_k, W_{\infty}))}^{\frac{1}{d}} \\
&\leq W_{\infty}^d \cdot N^{\frac{d-1}{d}}  \| \nabla f\|_{L^{\infty}}^{\frac{d-1}{d}} \left( \sum_{k=1}^{N}   \| \nabla f \|_{L^1(B(x_k, W_{\infty}))}\right)^{\frac{1}{d}}.
\end{align*}
Applying Lemma 3 as in the proof of the main result above leads to
$$ \left( \sum_{k=1}^{N}   \| \nabla f \|_{L^1(B(x_k, W_{\infty}))}\right)^{\frac{1}{d}} \leq W_{\infty} \cdot N^{\frac{1}{d}} \cdot \| \nabla f\|_{L^1([0,1]^d)}^{\frac{1}{d}}.$$
\end{proof}

\end{document}